\documentclass[11pt,a4paper,reqno]{amsart}
\usepackage[english]{babel}
\usepackage[applemac]{inputenc}
\usepackage[T1]{fontenc}
\usepackage{palatino}
\usepackage{amsmath}
\usepackage{amssymb}
\usepackage{amsthm}
\usepackage{amsfonts}
\usepackage{graphicx}
\usepackage{mathtools}

\usepackage[colorlinks = true, citecolor = black]{hyperref}
\title{Curve packing and modulus estimates}
\author{Katrin F\"assler and Tuomas Orponen}
\address{University of Jyv\"askyl\"a, Department of Mathematics and Statistics}
\address{University of Helsinki, Department of Mathematics and Statistics}
\thanks{K.F. is supported by the Academy of Finland through the grant \emph{Sub-Riemannian manifolds from a quasiconformal viewpoint}, grant number 285159. T.O. is supported by the Academy of Finland through the grant \emph{Restricted families of projections and connections to Kakeya type problems}, grant number $274512$. T.O. is also a member of the Finnish CoE in Analysis and Dynamics Research.}
\email{katrin.s.fassler@jyu.fi}
\email{tuomas.orponen@helsinki.fi}
\subjclass[2010]{28A75 (Primary) 31A15, 60CXX (Secondary)}

\newcommand{\R}{\mathbb{R}}
\newcommand{\N}{\mathbb{N}}

\newcommand{\tn}{\mathbb{P}}

\newcommand{\calT}{\mathcal{T}}

\newcommand{\calD}{\mathcal{D}}
\newcommand{\calH}{\mathcal{H}}

\newcommand{\calE}{\mathcal{E}}
\newcommand{\calN}{\mathcal{N}}

\newcommand{\calS}{\mathcal{S}}

\newcommand{\E}{\mathbb{E}}

\newcommand{\diam}{\operatorname{diam}}
\newcommand{\card}{\operatorname{card}}

\numberwithin{equation}{section}

\theoremstyle{plain}
\newtheorem{thm}[equation]{Theorem}

\newtheorem{lemma}[equation]{Lemma}

\newtheorem{cor}[equation]{Corollary}
\newtheorem{proposition}[equation]{Proposition}

\theoremstyle{definition}

\theoremstyle{remark}
\newtheorem{remark}[equation]{Remark}

\begin{document}

\begin{abstract} A family of planar curves is called a \emph{Moser family} if it contains an isometric copy of every rectifiable curve in $\R^{2}$ of length one. The classical "worm problem" of L. Moser from 1966 asks for the least area covered by the curves in any Moser family. In 1979, J. M. Marstrand proved that the answer is not zero: the union of curves in a Moser family has always area at least $c$ for some small absolute constant $c > 0$. We strengthen Marstrand's result by showing that for $p > 3$, the $p$-modulus of a Moser family of curves is at least $c_{p} > 0$.
\end{abstract}

\maketitle

\section{Introduction}

The modulus of a curve family is a fundemental tool in the study of
quasiconformal mappings and in other areas of mapping theory, see
for instance the monograph  by O.~Martio et al.\ \cite{MR2466579}  for an overview. In a metric measure space $(X,\mu)$, the
\emph{$p$-modulus} of a family $\Gamma$ of arcs is the number
\begin{displaymath}
\mathrm{mod}_p(\Gamma) = \inf_{\rho\in \mathrm{adm}(\Gamma)} \int_X \rho^p \;\mathrm{d}\mu,
\end{displaymath}
where $\mathrm{adm}(\Gamma)$ is the collection of $\Gamma$-\emph{admissible} functions, namely
\begin{displaymath}
\mathrm{adm}(\Gamma)=\left\{\rho:X \to [0,\infty]\text{ Borel}:\int_{\gamma} \rho \;\mathrm{d}\mathcal{H}^1\geq 1 \text{ for all locally rectifiable }\gamma\in \Gamma\right\}.
\end{displaymath}
To obtain an upper bound for the modulus of a given family, it is
sufficient to find one appropriate admissible density
$\rho$, but an estimate from below requires a lower bound for the $L^{p}(\mu)$-norms of \emph{all} admissible densities. To find an optimal lower bound, or even to show that the modulus of a curve
family is positive, is therefore often a challenging task. So far, this task has mainly been performed for families of curves which either (i) foliate some domain (in this case a non-vanishing modulus corresponds to a Fubini-type theorem), or (ii) consist of all curves connecting two given continua (the $p$-modulus of such a curve family coincides with the $p$-capacity of the said pair of continua).

In this paper, $\mu$ is Lebesgue measure in $X = \R^{2}$, and we consider certain curve families that are quite far from either type (i) or (ii), namely ones that arise from a classical \emph{curve packing problem}. The prototypical example of a curve packing problem is Kakeya's question from the early 1900's: if a set $K \subset \R^{2}$ contains a translate of every unit line segment in $\R^{2}$, what is the minimal (or infimal) area of $K$? The famous answer, due to Besicovitch \cite{MR1544912}, is "zero", and indeed there are compact \emph{Besicovitch sets} of vanishing measure, which satisfy Kakeya's condition.

For the moment, a curve family $\Gamma$ containing a translate of every unit line segment in $\R^{2}$ will be called a \emph{Kakeya family}. With this terminology, Besicovitch's result can be rephrased by saying that the curves in a Kakeya family $\Gamma$ need not cover a positive area, and in particular $\Gamma$ need not have positive $p$-modulus for any $1\leq  p < \infty$. Indeed, if $K$ is a Besicovitch set, then $\chi_{K} \in \textrm{adm}(\Gamma)$ for a certain Kakeya family $\Gamma$, yet $\|\chi_{K}\|_{L^p} = 0$ for every $1\leq p < \infty$. The same conclusion holds, if the line segments are replaced by $n$-sided polygons, see \cite{MR0264035}, or even circular arcs, see \cite{MR0229779,MR1535168,MR0297953}.

In short, if $\Gamma$ contains a translate of every curve in some rather small collection of initial suspects -- such as line segments, or circular arcs -- there is little hope of positive modulus. So, for positive results, the collection of suspects needs to be enlarged, and a natural candidate for is the collection of \textbf{all} plane curves of length one. Indeed, around 1966, L. Moser \cite{MR664434} proposed\footnote{The book \cite{MR664434} is from 1980 and hence not the original reference for Moser's question, which made its first appearance in an unpublished, mimeographed problem list entitled "Poorly formulated unsolved problems of combinatorial geometry".} the following question: if $\Gamma$ is a family of curves containing an isometric copy of every plane curve of length one, then what is the minimal area covered by the curves in $\Gamma$?  The question, known as "Moser's worm problem", has attracted considerable interest in computational geometry. As far as we know, the best upper bound known to date is due to Norwood and Poole \cite{MR1961007}, showing that a \emph{Moser family} of curves need not cover an area larger than $\approx 0.260437$ in general. We refer to the monograph by P.~Brass et al.\ \cite{MR2163782}, Section 11.4, for a bibliographical overview. From our point of view, however, more interesting is a theorem of J. M. Marstrand \cite{MR557120} from 1979, which states that the answer to Moser's question is not zero: the curves in a Moser family $\Gamma$ always cover a positive area, and a quantitative (if very small) lower bound for the measure can be extracted from Marstrand's argument.

Encouraged by Marstrand's result, one could hope that a Moser family of curves has positive $p$-modulus for some $1\leq  p < \infty$. This cannot happen for $1\leq  p \leq 2$, however: all the curves in a Moser family can contain the origin, and even the family of \textbf{all} curves containing the origin has vanishing $2$-modulus, see Corollary 7.20 in Heinonen's book \cite{MR1800917}. Our main result states that the $p$-modulus is non-vanishing for all $p > 3$; the cases $2 < p \leq 3$ remain open.

\begin{thm}\label{main} Let $\Gamma$ be a family of curves which contains an isometric copy of every set of the form
\begin{displaymath} G_{f} := \{(x,f(x)) : x \in [0,1]\}, \end{displaymath}
where $f \colon [0,1] \to [0,1]$ is $1$-Lipschitz. Then $\mathrm{mod}_{p}(\Gamma) \geq c > 0$ for every $p > 3$, where $c > 0$ is a constant depending only on $p$.
\end{thm}
We recover Marstrand's theorem, or in fact a slightly stronger version:
\begin{cor}\label{mainCor} Let $\delta \in (0,1]$, and associate to every length-$1$ rectifiable curve $\gamma$ in $\R^{2}$ an $\calH^{1}$-measurable subset $E_{\gamma}$ of length at least $\delta$, and an isometry $\iota_{\gamma}$. Then, the union of the sets $\iota_{\gamma}(E_{\gamma})$ has Lebesgue outer measure at least $\gtrsim c_{p}\delta^{p}$ for any $p > 3$. \end{cor}

\subsection{Notational conventions} A closed disc of radius $r > 0$ and centre $x \in \R^{2}$ is denoted by $B(x,r) \subset \R^{2}$. The notation $A \lesssim B$ means that $A \leq CB$ for some absolute constant $C \geq 1$, and the two-sided inequality $A \lesssim B \lesssim A$ is abbreviated to $A \sim B$. Throughout the text, we use the letter $C$ to denote a (large) constant, whose value may change from one occurrence to the next.

The Lebesgue outer measure of an arbitrary set $A \subset \R^{2}$ is denoted by $|A|$, and we often write "area" instead of "Lebesgue measure". One-dimensional Hausdorff measure in $\R^{2}$ is denoted by $\calH^{1}$; for the definition and basic properties of Hausdorff measures, see Mattila's book \cite{MR1333890}.

\subsection{Proof sketch, and the structure of the paper} We close the introduction with a quick overview of the paper. In Section \ref{probabilitySpace}, we define a large family of Lipschitz graphs $G(\omega)$ parametrised by a probability space $\Omega \ni \omega$. In Section \ref{probabilityLemmas}, a sequence of three lemmas establishes that given a fixed small set $E \subset \R^{2}$ of area $|E| \leq \epsilon$, it is highly unlikely that the intersection of $G(\omega)$ with any isometric copy of $E$ should have $\calH^{1}$-measure far larger than $\epsilon^{1/3}$. The exponent $1/3$ is responsible for the restriction $p > 3$ in Theorem \ref{main}. After these preparations, the proof of Theorem \ref{main}, contained in Section \ref{mainProof}, is fairly quick: if $\rho \in \operatorname{adm}(\Gamma)$, we write $\rho \sim \sum 2^{j} \chi_{E_{j}}$, where $E_{j} = \{x : \rho \sim 2^{j}\}$. Assuming that $\|\rho\|_{L^{p}}^{p} < \epsilon$ for some $p > 3$ and small $\epsilon > 0$, the areas $|E_{j}|$, $j \in \N$, decay faster than $2^{-3j}$, and it is unlikely that the intersection of $G(\omega)$ with any isometric copy of any $E_{j}$ has length $\sim 2^{-j}$. Consequently, we find a graph $G(\omega)$ such that $\calH^{1}(G(\omega) \cap \iota(E_{j})) \ll 2^{-j}$ for any isometry $\iota$, and any $j$. This means that the $\calH^{1}$-integral of $\rho$ over any isometric copy of $G(\omega)$ falls short of $1$, and the ensuing contradiction gives a lower bound for $\epsilon$.

The final section of the paper contains the proof of Corollary \ref{mainCor} and some further remarks.

\section{A probability space of Lipschitz graphs}\label{probabilitySpace}

\subsection{Families of parallelograms}\label{parallelogramCollection} Let $(m_{k})_{k \in \N}$ be an non-decreasing sequence of integers such that $m_{0} = 1$; write
\begin{displaymath} n_{k} := \prod_{j \leq k} m_{j}, \end{displaymath}
and assume that $\sum_{k \geq 1} 2^{k}/n_{k} < 1/3$. We consider a
space of random Lipschitz graphs in $[0,1]^{2}$ constructed in the
following way. For each "generation" $k$, we define a (random)
family $\calT_{k}$ of $2^{k}$ increasingly thin and long closed
parallelograms $\calT_{k} := \{T^{k}_{1},\ldots,T^{k}_{2^{k}}\}$
with the following properties:
\begin{itemize}
\item[(i)] The parallelograms in $\calT_{k}$ are all contained in $[0,1]^{2}$.
\item[(ii)] Two sides of the parallelograms are parallel to the $y$-axis; these will be referred to as the "vertical" sides.
\item[(iii)] The base and height of every parallelogram in $\calT_{k}$ are $2^{-k}$ and $1/n_{k}$, respectively, so that $|T| = 2^{-k}/n_{k}$ for every $T \in \calT_{k}$.
\item[(iv)] For $1 \leq j < 2^{k}$, the right vertical side of $T_{j}$ coincides with the left side of $T_{j + 1}$.
\item[(v)] Fix $\lambda =1/2$. Roughly $2^{\lambda k}$ of the parallelograms in $\calT_{k}$ are called \emph{exceptional}, and the rest are \emph{normal}. The collections of exceptional and normal parallelograms are denoted by $\calE_{k}$ and $\calN_{k}$, respectively. The exceptional parallelograms are quite evenly distributed: between every consecutive pair of parallelograms $T^{k}_{i},T^{k}_{j} \in \calE_{k}$, there are roughly $2^{k(1-\lambda)}$ parallelograms in $\calN_{k}$. Moreover, no exceptional parallelogram has common boundary with $[0,1]^{2}$.
\end{itemize}

Heuristically, it follows from (i)--(iv) that the union of the parallelograms in $\calT_{k}$ roughly forms the $(1/n_{k})$-neighbourhood of the graph of a continuous function $f_{\calT_{k}} \colon [0,1] \to [0,1]$. It will later be shown that all functions $f_{\calT_{k}}$ obtained this way are Lipschitz.

To begin the construction, let $\calT_{0} := [0,1]^{2}$. We declare the only element of $\calT_{0}$ normal, so there are no exceptional parallelograms at the $0^{th}$ level. Then, assume that $\calT_{k}$, $\calE_{k}$ and $\calN_{k}$ have already been defined for some $k \geq 0$. To define the family $\calT_{k + 1}$, consider any maximal "string" of consecutive normal parallelograms (that is, a maximal collection of normal parallelograms with the property that the union is connected). By (v), this collection consists of roughly $2^{(1-\lambda)k}$ parallelograms with base $2^{-k}$ and height $1/n_{k}$. Subdivide each of them to a "pile" of $m_{k}$ parallelograms with base $2^{-k}$ and height $(1/n_{k}) \cdot (1/m_{k}) = 1/n_{k + 1}$ in the obvious way, see Figure \ref{fig2}.
\begin{figure}[h!]
\begin{center}
\includegraphics[scale = 0.6]{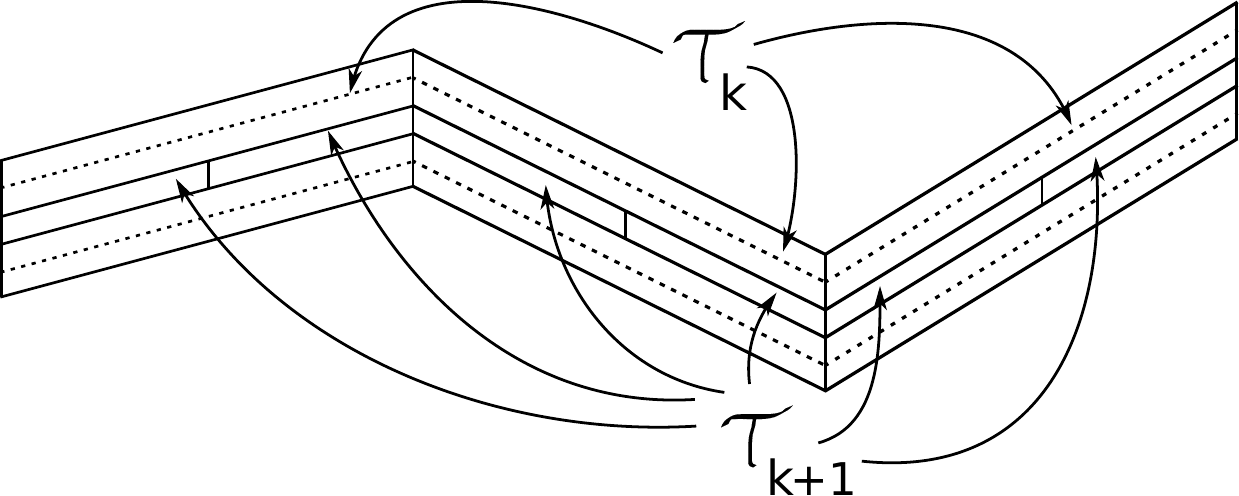}
\caption{A string of three normal parallelograms in $\calT_{k}$,
and how to construct $\calT_{k + 1}$ inside them. Here $m_{k} =
5$.}\label{fig2}
\end{center}
\end{figure}

Then, pick a number $j$ at random in $\{1,\ldots,m_{k}\}$, and, from each "pile", pick the $j^{th}$ parallelogram (so the number $j$ is common for this particular string of normal parallelograms). The new parallelograms obtained in this manner clearly satisfy (i)--(iv), except that the length of their base is twice too long. To remedy this, we simply cut the parallelograms in half with vertical lines.

We repeat the procedure inside every maximal string of normal parallelograms, that is, roughly $2^{\lambda k}$ times. On each occasion, the random integer in $\{1,\ldots,m_{k}\}$ is re-selected independently of previous choices. Now, the construction of $\calT_{k + 1}$ is nearly complete: we only need to specify what to do inside the parallelograms in $\calE_{k}$. Let $T := T_{j}^{k} \in \calE_{k}$. Then, $T$ is adjacent to two normal parallelograms $T'$ and $T''$. The construction of $\calT_{k + 1}$ inside $T'$ and $T''$ is already finished, so, for instance, $T'$ contains a parallelogram in $\calT_{k + 1}$, whose right vertical side $V'$ is contained in the right vertical side of $T'$. The same is true of $T''$, with "right" replaced by "left", and $V'$ replaced by $V''$. Now, there is a unique parallelogram inside $T$ with base $2^{-k}$ and height $1/n_{k + 1}$, which "connects" $V'$ to $V''$, see Figure \ref{fig3}.
\begin{figure}[h!]
\begin{center}
\includegraphics[scale = 1.0]{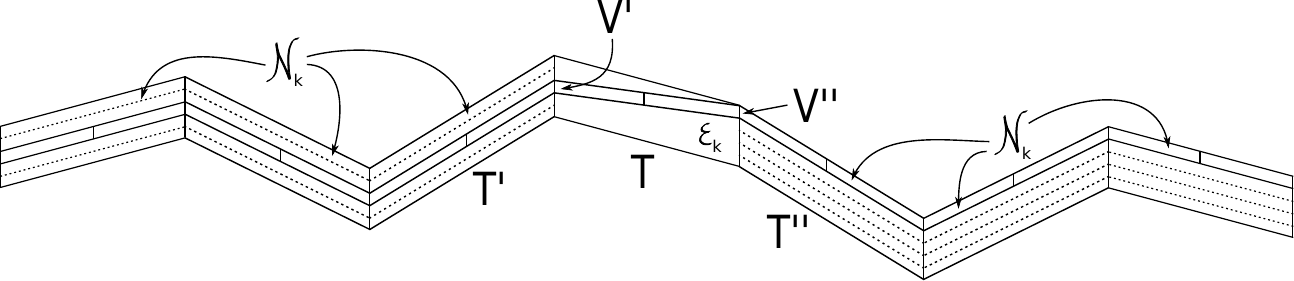}
\caption{The construction of $\calT_{k + 1}$ inside a
parallelogram in $\calE_{k}$.}\label{fig3}
\end{center}
\end{figure}
We split this parallelogram into two halves with a vertical line through the middle, and the ensuing two parallelograms are added to $\calT_{k + 1}$. Repeating this procedure inside each parallelogram in $\calE_{k}$, the construction of $\calT_{k + 1}$ is complete.

Finally, we need to specify $\calE_{k + 1}$ and $\calN_{k + 1}$. To do this, we declare the first $\sim 2^{(1-\lambda)k}$ leftmost parallelograms in $\calT_{k + 1}$ to be normal. The next one is exceptional. Then, again $\sim 2^{(1-\lambda)k}$ normal parallelograms, followed by one exceptional. Continue this until the right vertical side of $[0,1]^{2}$ is reached. If, now, the rightmost parallelogram happens to end up in $\calE_{k + 1}$, or the rightmost string of normal parallelograms contains far fewer than $2^{(1-\lambda)k}$ sets, re-adjust the cardinalities of the previous strings slightly so that the rightmost parallelogram belongs to $\calN_{k + 1}$, and every string contains $\sim 2^{(1-\lambda)k}$ normal parallelograms.

\subsection{A space of Lipschitz graphs} Let $(\calT_{k})_{k \in \N}$ be a sequence of families of parallelograms such that $\calT_{k + 1}$ is obtained from $\calT_{k}$ via the preceding construction; we abbreviate this by saing that $\calT_{k + 1}$ is a \emph{child} of $\calT_{k}$. The unions
\begin{displaymath} G_{\calT_{k}} := \bigcup_{j = 1}^{2^{k}} T^{k}_{j}, \qquad k \geq 0, \end{displaymath}
then form a nested sequence of compact sets inside $[0,1]$, which converge to the graph of a certain $[0,1]$-valued function $f$ defined on the interval $[0,1]$. We now prove that such functions $f$ are Lipschitz.
\begin{lemma}[Lipschitz lemma]\label{l:Lip} If $\sum_{k=1}^{\infty} 2^{k}/n_{k} <1/3 $, then any function of the form $f$ above is $1$-Lipschitz.
\end{lemma}

\begin{proof}
Two edges of the parallelograms used in the definition of
$G_{\calT_{k}}$ are formed by vertically translated graphs of
linear functions over segments in the $x$-axis. While the slope of
these graphs can increase from one generation to the next, the
summability condition on $(n_k)_k$, and the fact that the first generation slope is zero, ensures that the corresponding
functions have bounded derivatives. To make this precise, we first
compute the slopes $a_j^{k+1}$ that can occur in generation $k+1$.

By construction, the slopes in the first two generations, namely $a_j^0$ and $a_j^1$, can be assumed to be zero.

Then the two parallelograms $T_{j'}^{k+1}$ which we select inside a
given $T_j^k$, are bounded by linear graphs, all of which have the
same slope $a_{j'}^{k+1}$. In case $T_j^k$ is normal, then
$a_{j'}^{k+1}=a_j^k$.

If $T_j^k$ is exceptional, the slope can increase, but only by a
controlled amount:
\begin{displaymath}
a_{j'}^{k+1} = a_j^k + \frac{m-n}{n_{k+1} 2^{-k}}
\end{displaymath}
for suitable $m,n \in \{0,\ldots,m_{k+1}-1\}$. Thus,
\begin{displaymath}
\left|\frac{m-n}{n_{k+1} 2^{-k}}\right|\leq
\frac{m_{k+1}-1}{m_{k+1}n_k 2^{-k}}=
\frac{2^k}{n_k}-\frac{2^k}{n_{k+1}},
\end{displaymath}
and hence,
\begin{equation}\label{eq:bound_slope}
\sup_{1\leq j\leq 2^{k+1}}|a_{j'}^{k+1}| \leq \sum_{l=1}^{\infty}
\frac{2^l}{n_l}<1/3.
\end{equation}
Given this information, we proceed to prove the $1$-Lipschitz
continuity of $f$.

Let $(x,y)$ and $(x',y')$ be two elements in the graph of $f$.
 Without loss of generality, we
may assume that $x\neq x'$. Then there exists a unique
non-negative integer $k$ such that
\begin{displaymath}
\frac{1}{2^{k+1}} < |x-x'|\leq \frac{1}{2^k}.
\end{displaymath}
It follows that either there is a parallelogram of $\calT_{k}$
containing both $(x,y)$ and $(x',y')$, or the two points lie in
two adjacent parallelograms belonging to $\calT_{k}$.

Assume first that the two points lie in the same parallelogram
$T_j^k$. Then we have the following estimate:
\begin{align*}
|f(x)-f(x')|\leq |a_j^k||x-x'| + n_k^{-1}&\leq \left(\sup_{1\leq
j\leq 2^k}|a_j^k|+n_k^{-1}|x-x'|^{-1}\right)|x-x'|\\
&\leq \left(\sup_{1\leq j\leq
2^k}|a_j^k|+n_k^{-1}2^{k+1}\right)|x-x'|.
\end{align*}
The assumptions on $(n_k)_k$ together with \eqref{eq:bound_slope}
settle the case for $(x,y)$ and $(x',y')$ lying in the same
parallelogram of generation $k$.

Assume next that the two points lie in adjacent parallelograms $T^{k}_{j}$ and $T^{k}_{j + 1}$. Without
loss of generality, $x<x''<x'$, where $x''$ is the first coordinate of
the common vertical edge of $T^{k}_{j}$ and $T^{k}_{j + 1}$. Then, by the
connectedness of the union $T^{k}_{j} \cup T^{k}_{j + 1}$,
\begin{align*}
|f(x)-f(x')|&\leq \frac{2}{n_k} + |a_j^k||x-x''|+
|a_{j+1}^k||x''-x'|\leq \frac{2}{n_k} + \sup_{1\leq j\leq 2^k}
|a_j^k| |x-x'|,
\end{align*}
and we conclude by the same reasoning as before.
\end{proof}

With the previous lemma in mind, we see that any sequence $(\calT_{k})_{k \in \N}$, where $\calT_{k + 1}$ is a child of $\calT_{k}$, can be identified with a Lipschitz graph in $[0,1]^{2}$. In essence, we wish to define a probability measure on the space of all possible graphs so obtained, but in practice it is slightly easier to work with probabilities in the space of all sequences $(\calT_{k})_{k \in \N}$.

Let $\Omega$ be the space of all such sequences. Then $\Omega$ can be viewed as the set of infinite branches in a tree rooted at $\calT_{0}$. By definition, the only vertex of height $0$ is $\calT_{0}$, and the vertices of height $k + 1$ are obtained by considering each height-$k$ vertex $\calT_{k}$, and adding a vertex for each of its children $\calT_{k + 1}$. As a technical point,
it is conceivable that some fixed collection $\calT_{k}$ corresponds to several vertices in the tree, in case $\calT_{k}$ can be obtained via several different sequences starting from $\calT_{0}$ (this is most likely not possible in practice, but even if it is, nothing changes below).

The space $\Omega$ supports a natural probability measure $\tn$: assume that $k_{0} \geq 0$ is given, and there are $N_{k_{0}} \in\N$ finite sequences of the form $(\calT_{k})_{k = 0}^{k_{0}}$ (where $\calT_{k + 1}$ is a child fo $\calT_{k}$). Then, if $(\calT_{k}')_{k = 0}^{k_{0}}$ is any one of these finite sequences, $\tn$ assigns probability $1/N_{k_{0}}$ to the set of infinite sequences
\begin{displaymath} \{(\calT_{k})_{k \in \N} : \calT_{k} = \calT_{k}' \text{ for all } 0 \leq k \leq k_{0} \}. \end{displaymath}
The numbers $N_{k}$ grow very rapidly: in fact
\begin{displaymath} N_{k + 1} = N_{k} \cdot m_{k}^{C_{k}}, \qquad k \in \N, \end{displaymath}
with $C_{k} \sim 2^{\lambda k}$, but this is rather irrelevant for us.

From now on, we denote the generic element of $\Omega$ by $\omega$. Thus, every $\omega \in \Omega$ is a sequence of the form $(\calT_{k})_{k \in \N}$, and we write $\omega_{k} := \calT_{k}$. If $\calT_{k}'$ is a fixed parallelogram collection of generation $k$, it then makes sense to speak of events of the form $\{\omega \in \Omega : \omega_{k} = \calT_{k}'\}$, and indeed events of this form partition $\Omega$, as $\calT_{k}'$ ranges over all possible parallelogram collections of generation $k$. As discussed earlier, every sequence $\omega \in \Omega$ corresponds to a Lipschitz graph contained in $[0,1]^{2}$, and we denote this graph by $G(\omega)$; thus, if $\omega = (\calT_{k})_{k \in \N}$, then
\begin{displaymath} G(\omega) = \bigcap_{k \in \N} G_{k}(\omega), \end{displaymath}
where $G_{k}(\omega)$ is the "level $k$ approximation" $G_{k}(\omega) := G_{\calT_{k}}$.

\section{Lemmas on intersections}\label{probabilityLemmas}

From now on, a generic isometry in $\R^{2}$ will be denoted by $\iota$. Given a measurable set $E \subset \R^{2}$ with $|E| > 0$, we are interested in bounding the $\tn$-probability that $G(\omega) \cap \iota(E)$ has large $\calH^{1}$-measure for some isometry $\iota$. This will accomplished in Lemma \ref{densityLemmaThree} (far) below. However, to get our hands on $\calH^{1}(G(\omega) \cap \iota(E))$, we first need to study a sequence of intermediate quantities, namely the \emph{densities of $\iota(E)$ inside $G_{k}(\omega)$}, $k \in \N$. More generally, if $A,B \subset \R^{2}$ are two Borel sets with $|B| > 0$, the \emph{density of $A$ inside $B$} is
\begin{displaymath} D_{A}(B) := \frac{|A \cap B|}{|B|}. \end{displaymath}

For the rest of the paper, we fix the sequence $(m_{k})_{k \in \N}$ introduced in the previous section. Any sequence so that
\begin{displaymath} 100 {k^{2} \cdot 2^{k}} \leq n_{k} = \prod_{j \leq k} m_{j} \leq 10000 k^{2} \cdot 2^{k}. \end{displaymath}
will do. In particular, then the graphs we are considering are $1$-Lipschitz by Lemma \ref{l:Lip}.

Before stating the first probabilistic lemma, we formulate a geometric one:
\begin{lemma}[Continuity lemma]\label{continuityLemma}
Let $T$ be a parallelogram in $\mathcal{T}_k$, and $E$ a measurable set in $\mathbb{R}^2$. Then, for all isometries $\iota_{1},\iota_{2}$,
\begin{displaymath}
|T\cap \iota_{1}(E)|-|T\cap \iota_{2}(E)| \lesssim 2^{-k}\|\iota_{1} - \iota_{2}\|,
\end{displaymath}
where $\| \cdot \|$ stands for operator norm $\|L\| := \sup \{ |Lx| : |x| \leq 1\}$.
\end{lemma}

\begin{proof} Denote the $\delta$-neighbourhood of a set $A \subset \R^{2}$ by $N(A,\delta)$. It follows from the Lipschitz lemma, or rather its proof, that the parallelogram $T$ cannot be too tilted: all of its sides have length $\lesssim 2^{-k}$; in particular $|N(\partial T,\delta)| \lesssim \delta \cdot 2^{-k}$ for $0 < \delta \leq 2^{-k}$.

The lemma is clear for $2^{-k} \leq \|\iota_{1} - \iota_{2}\|$, because $|T| \lesssim 2^{-2k}$. So, we assume that $0 < \|\iota_{1} - \iota_{2}\| \leq 2^{-k}$. Now Lemma 8 in Marstrand's paper \cite{MR557120} says that $|\iota_{1}(T) \setminus \iota_{2}(T)| \leq |N(\partial T,\|\iota_{1} - \iota_{2}\|)| \lesssim 2^{-k}\|\iota_{1} - \iota_{2}\|$. Consequently,
\begin{align*} |T \cap \iota_{1}(E)| - |T \cap \iota_{2}(E)| & = |\iota_{1}^{-1}(T) \cap E| - |\iota_{2}^{-1}(T) \cap E|\\
& \leq |\iota_{1}^{-1}(T) \setminus \iota_{2}^{-1}(T)| + |\iota_{2}^{-1}(T) \cap E| - |\iota_{2}^{-1}(T) \cap E|\\
& \lesssim 2^{-k}\|\iota_{1} - \iota_{2}\|,  \end{align*}
as claimed. \end{proof}

\begin{lemma}[Density lemma]\label{densityLemmaOne} Fix $\epsilon, \kappa > 0$, and let $E$ be a Borel set with $|E| \leq \epsilon$ and $\diam E \leq 2$. Then
\begin{displaymath} \tn\left\{\sup_{k \in \N} \sup_{\iota} D_{\iota(E)}(G_{k}(\omega)) \geq \epsilon^{1/3 - \kappa} \right\} \leq \epsilon^{3}, \end{displaymath}
if $\epsilon > 0$ is small enough, depending only on $\kappa$.
\end{lemma}

\begin{remark} A quick word on the numerology. The threshold $1/3$ has a real meaning: if we could replace it by $1/q$ for some $q < 3$, then we could prove Theorem \ref{main} for $p > q$ instead of $p > 3$. In particular, the lemma ceases to be true for $q < 2$, since Theorem \ref{main} fails for $p = 2$. On the other hand, the exponent $3$ on the right hand side of the inequality is arbitrary; one could replace it by any number $k > 1$ by requiring $\epsilon$ to be even smaller.    \end{remark}

\begin{proof}[Proof of Lemma \ref{densityLemmaOne}]
We may assume that $\kappa \leq 1/3$.

Fix $\omega = (\calT_{k})_{k \in \N}$ and write $G_{k} := G_{k}(\omega)$. We start with the trivial estimate
\begin{equation}\label{form4} D_{\iota(E)}(G_{k}) = \frac{|G_{k} \cap \iota(E)|}{|G_{k}|} \leq \frac{|E|}{|G_{k}|} \leq n_{k}\epsilon \leq 10000k^{2} \cdot 2^{k}\epsilon. \end{equation}
In particular, for any isometry $\iota$, we have $D_{\iota(E)}(G_{k}) < \epsilon^{1/3 - \kappa}/2$ as long as
\begin{equation}\label{eq:condkepsilon} 10000k^{2} \cdot 2^{k}\epsilon < \epsilon^{1/3 - \kappa}/2, \end{equation}
which is true if $2^{k} \leq \epsilon^{-(2 + \kappa)/3}$, and assuming that $\epsilon > 0$ is small enough depending on $\kappa$. Consequently, assuming that $\epsilon^{-(2 + \kappa)/3} = 2^{k_{\epsilon}}$ for some $k_{\epsilon} \in \N$, we have
\begin{equation}\label{form6} \sup_{k \leq k_{\epsilon}} \sup_{\iota} D_{\iota(E)}(G_{k}(\omega)) < \frac{\epsilon^{1/3 - \kappa}}{2}, \quad \omega \in \Omega. \end{equation}
Let $r_{k_{\epsilon}} := \epsilon^{1/3 - \kappa}/2$, and
\begin{displaymath} r_{k} := r_{k - 1} + \frac{\epsilon^{1/3 - \kappa}}{2(k - k_{\epsilon} + 1)^{2}}, \qquad k > k_{\epsilon}. \end{displaymath}
Observe that $r_{k} \nearrow \pi^{2}\epsilon^{1/3 - \kappa}/12 < \epsilon^{1/3 - \kappa}$. So, if $\sup_{k} \sup_{\iota} D_{\iota(E)}(G_{k}(\omega)) \geq \epsilon^{1/3 - \kappa}$, \eqref{form6} implies that there exists $k > k_{\epsilon}$ such that $\sup_{\iota} D_{\iota(E)}(G_{k}(\omega)) \geq r_{k}$. In particular, there exists a smallest $k$ such that this happens. Thus, writing
\begin{displaymath} g_{k}(\omega) := \sup_{\iota} D_{\iota(E)}(G_{k}(\omega)), \end{displaymath}
we have
\begin{equation}\label{form14} \tn\{\sup_{k} g_{k}(\omega) \geq \epsilon^{1/3 - \kappa}\} \leq \sum_{k = k_{\epsilon} + 1}^{\infty} \tn\{g_{k - 1}(\omega) < r_{k - 1} \text{ and } g_{k}(\omega) \geq r_{k}\}. \end{equation}

We now fix $k > k_{\epsilon}$ and estimate the term with index $k$ in \eqref{form14}. The value of the function $g_{k - 1}$ only depends on the collection $\calT_{k - 1}$. So, the event we are interested in, namely $\{g_{k - 1}(\omega) < r_{k - 1} \text{ and } g_{k}(\omega) \geq r_{k}\}$ can be partitioned into events of the form
\begin{displaymath} \{G_{k - 1}(\omega) = G_{\calT_{k - 1}} \text{ and } g_{k}(\omega) \geq r_{k}\}, \end{displaymath}
where $\calT_{k - 1}$ is some fixed collection of parallelograms with the property that
\begin{equation}\label{form7} \sup_{\iota} D_{\iota(E)}(G_{\calT_{k - 1}}) < r_{k - 1}. \end{equation}
Then, we write
\begin{align} \tn & \{G_{k - 1}(\omega) = G_{\calT_{k - 1}} \text{ and } g_{k}(\omega) \geq r_{k}\} \notag\\
&\label{form13} = \tn\{G_{k - 1}(\omega) = G_{\calT_{k - 1}}\} \cdot \tn\{g_{k}(\omega) \geq r_{k} \mid G_{k - 1}(\omega) = G_{\calT_{k - 1}}\}, \end{align}
and focus on estimating the conditional probabilities $\tn\{g_{k}(\omega) \geq r_{k} \mid G_{k - 1}(\omega) = G_{\calT_{k - 1}}\}$.
We fix the collection $\calT_{k - 1}$ satisfying \eqref{form7}, and we abbreviate the conditional probability $\tn\{ \cdot \mid G_{k - 1}(\omega) = G_{\calT_{k - 1}}\}$ to $\tn_{k}\{ \cdot \}$. Then, we also fix an isometry $\iota_{0}$. The first step is to estimate the $\tn_{k}$-probability of the event
\begin{equation}\label{form8} D_{\iota_{0}(E)}(G_{k}(\omega)) \geq r_{k - 1} + \frac{\epsilon^{1/3 - \kappa}}{4(k - k_{\epsilon} + 1)^{2}}. \end{equation}
For notational convenience, we assume $\iota_{0} = \operatorname{Id}$. The $\tn_{k}$-probability of the event in \eqref{form8} is quite tractable: the approximate graph $G_{k}(\omega)$ only depends on the parallelogram collection $\calT_{k}$, which is constructed inside the fixed collection $\calT_{k - 1}$ using the random process described in Section \ref{parallelogramCollection}.

Let $\calN_{k - 1}$ and $\calE_{k - 1}$ be the collections of normal and exceptional parallelograms of $\calT_{k - 1}$, respectively. Recall that $\calE_{k - 1}$ contains $\sim 2^{\lambda k}$ parallelograms, and that the parallelograms in $\calN_{k - 1}$ can be partitioned into $\sim 2^{\lambda k}$ "strings" of parallelograms, each containing $\sim 2^{(1-\lambda)k}$ elements in $\calT_{k - 1}$ (recall that $\lambda = 1/2$). We denote the collection of such strings by $\calS_{k - 1}$, so that every set $S \in \calS_{k - 1}$ is a union of $\sim 2^{(1-\lambda)k}$ consecutive sets in $\calT_{k - 1}$. Now, for any child $\calT_{k}$ of $\calT_{k - 1}$, we have
\begin{displaymath} |G_{\calT_{k}} \cap E| = \sum_{S \in \calS_{k - 1}} |S \cap G_{\calT_{k}} \cap E| + \sum_{T \in \calE_{k - 1}} |T \cap G_{\calT_{k}} \cap E|. \end{displaymath}
For the second sum, we use the trivial estimate $|T \cap G_{\calT_{k}} \cap E| \leq |T| \sim 2^{-2k}/k^{2}$, which gives
\begin{align*} \frac{1}{|G_{\calT_{k}}|} \sum_{T \in \calE_{k - 1}} |T \cap G_{\calT_{k}} \cap E| & \lesssim \frac{2^{k(\lambda - 2)}/k^{2}}{2^{-k}/k^{2}} = 2^{-k/2} = 2^{-k_{\epsilon}/2}2^{(k_{\epsilon} - k)/2}\\
& = \epsilon^{(2 + \kappa)/6}\cdot 2^{(k_{\epsilon} - k)/2} = \epsilon^{7\kappa/6} \cdot \epsilon^{1/3 - \kappa} \cdot 2^{(k_{\epsilon} - k)/2}. \end{align*}
It follows that if $\epsilon > 0$ is small enough, depending on $\kappa$, we have
\begin{displaymath} \frac{1}{|G_{\calT_{k}}|} \sum_{T \in \calE_{k - 1}} |T \cap G_{\calT_{k}} \cap E| < \frac{\epsilon^{1/3 - \kappa}}{8(k - k_{\epsilon} + 1)^{2}}, \quad k > k_{\epsilon}. \end{displaymath}
Thus, if \eqref{form8} holds, we must have
\begin{displaymath} \frac{1}{|G_{\calT_{k}}|} \sum_{S \in \calS_{k - 1}} |S \cap G_{\calT_{k}} \cap E| \geq r_{k - 1} + \frac{\epsilon^{1/3 - \kappa}}{8(k - k_{\epsilon} + 1)^{2}}. \end{displaymath}
Assume for convenience that all the strings in $\calS_{k - 1}$ have the same measure.\footnote{This can be easily arranged for all but one of the strings during the construction, and then the total area of the one remaining "bad" string is so small, at most $\lesssim 2^{-3k/2}/k^{2}$, that it can be added to the exceptional parallelograms in the above estimation.} Then, also the sets $S \cap G_{\calT_{k}}$ have the same measure irregardless of the choice of $\calT_{k}$, and we denote this quantity by $|S \cap G_{\calT_{k}}|$. Now $(\card \calS_{k - 1}) \cdot |S \cap G_{\calT_{k}}| \leq |G_{\calT_{k}}|$, so the previous inequality implies that
\begin{equation}\label{form9} \frac{1}{\card \calS_{k - 1}} \sum_{S \in \calS_{k - 1}} \frac{|S \cap G_{\calT_{k}} \cap E|}{|S \cap G_{\calT_{k}}|} \geq r_{k - 1} + \frac{\epsilon^{1/3 - \kappa}}{8(k - k_{\epsilon} + 1)^{2}}.\end{equation}
The left hand side of \eqref{form9} can be interpreted as the average of the random variables
\begin{displaymath} X_{S} := \frac{|S \cap G_{\calT_{k}} \cap E|}{|S \cap G_{\calT_{k}}|}, \qquad S \in \calS_{k - 1}. \end{displaymath}
The set $E$ and the collection $\calT_{k - 1}$ being fixed, it follows from the construction that the variables $X_{S}$ are independent. They take values in $[0,1]$, and the expectation of $X_{S}$ is
\begin{equation}\label{form10} \E_{k}[X_{S}] = \frac{|E \cap S|}{|S|} =: d_{S}, \quad S \in \calS_{k - 1}. \end{equation}
This follows from the fact that the $m_{k}$ possible sets $S \cap G_{\calT_{k}}$ partition $S \cap E$
(except for zero measure boundaries), and the quantity $S \cap G_{\calT_{k}}$ does not depend on the choice of $\calT_{k}$.

Since we can rather easily estimate the \emph{expected value} of
average of the random variables $X_S$, $S\in \mathcal{S}_{k-1}$,
we wish to ensure that the empirical mean does not deviate too
much from this expected average. Such an estimate is provided by
Hoeffding's inequality, which we apply below after some
preparations. We quickly recall the inequality for the reader's convenience:
\begin{proposition}[Hoeffding's inequality \cite{MR0144363}]
Let $X_1,\ldots, X_n$ be independent random variables such that $a_i \leq  X_i \leq  b_i$ almost surely. Then, for $t > 0$,
\begin{equation}\label{eq:Hoeffding2}
\mathbb{P} ( \overline{X}- \mathbb{E}(\overline{X})\geq t)\leq \exp\left(-\frac{2 n^2 t^2}{\sum_{i=1}^n (b_i-a_i)^2}\right).
\end{equation}
\end{proposition}


We will apply the inequality to the random variables $X_{S}$, so we need to find $a_{S} \leq b_{S}$ such that $X_{S} \in [a_{S},b_{S}]$ almost surely. Clearly, $a_{S} = 0$ will do. To find $b_{S}$, we estimate
\begin{displaymath}
X_S = \frac{|S\cap G_{\mathcal{T}_k}\cap E|}{|S \cap G_{\mathcal{T}_k}|} \leq \min \left\{ 1, \frac{|S\cap E|}{|S \cap G_{\mathcal{T}_k}|}\right\}
\end{displaymath}
Since all normal strings have the same area, we find that
\begin{align*}
\mathrm{card} (\mathcal{S}_{k-1}) \cdot |S \cap G_{\mathcal{T}_k}| &= |G_{\mathcal{T}_k}|-\sum_{T \in \mathcal{E}_{k-1}} |T\cap G_{\mathcal{T}_k}|\gtrsim \frac{2^{-k}}{k^2}
\end{align*}
and thus
\begin{displaymath}
 |S \cap G_{\mathcal{T}_k}| \gtrsim  \frac{2^{-k}}{k^2} \frac{1}{\mathrm{card}(\mathcal{S}_{k-1})}.
\end{displaymath}
It follows that there exists a positive and finite constant $C$
such that $X_S$ is bounded from above by
\begin{displaymath}
b_S := \min \left\{1, C|S\cap E| \cdot k^2\cdot
2^{k}\mathrm{card}(\mathcal{S}_{k-1})\right\}.
\end{displaymath}
Therefore, recalling that $a_{S} = 0$, we have
\begin{align*}
\sum_{S\in \mathcal{S}_{k-1}} (b_S-a_S)^2 &= \sum_{S\in
\mathcal{S}_{k-1}} b_S^2\leq \sum_{S\in \mathcal{S}_{k-1}}
b_S \lesssim k^2 2^{k(1 + \lambda)} \sum_{S\in \mathcal{S}_{k-1}} |S\cap E| .
\end{align*}
In order to estimate this expression further from above, our task
is to find a good estimate for the sum $\sum|E\cap S|$. To this
end, recall that we are working with a fixed collection
$\mathcal{T}_{k-1}$ of parallelograms, which we have chosen so
that \eqref{form7} holds. In particular, we have
\begin{equation}\label{eq:parent_graph_nbhd_density}
|E\cap G_{\mathcal{T}_{k-1}}|< \epsilon^{1/3-\kappa}  |
G_{\mathcal{T}_{k-1}}|.
\end{equation}
Hence
\begin{displaymath}
\sum_{S\in \mathcal{S}_{k-1}} |E\cap S|\leq |E\cap
G_{\mathcal{T}_{k-1}}| <  \epsilon^{1/3-\kappa}  |
G_{\mathcal{T}_{k-1}}| \lesssim \epsilon^{1/3-\kappa} 2^{-k}/k^{2}
\end{displaymath}
and therefore
\begin{displaymath}
\sum_{S\in \mathcal{S}_{k-1}} (b_S-a_S)^2  \lesssim k^2 2^{k(1 + \lambda)} \sum_{S\in \mathcal{S}_{k-1}}  |S\cap E| \lesssim
2^{\lambda k} \epsilon^{1/3-\kappa}.
\end{displaymath}

We now apply Hoeffding's inequality \eqref{eq:Hoeffding2}, which gives here for $t>0$ that
\begin{equation}\label{form11} \tn_{k} \left\{ \frac{1}{\card \calS_{k - 1}} \sum_{S \in \calS_{k - 1}} X_{S} - \frac{1}{\card \calS_{k - 1}} \sum_{S \in \calS_{k - 1}} \E_{k}[X_{S}] \geq t  \right\} \leq \exp\left(\frac{-2(\card \calS_{k - 1})^2t^{2}}{C2^{\lambda k} \epsilon^{1/3-\kappa}}\right).\end{equation}

The average of the variables $X_{S}$ over $S \in \calS_{k - 1}$ is precisely the quantity in \eqref{form9} we are interested in. On the other hand, by \eqref{form10}, the average over the expectations $\E[X_{S}]$ equals
\begin{displaymath} \frac{1}{\card \calS_{k - 1}} \sum_{S \in \calS_{k - 1}} \E_{k}[X_{S}] = \frac{1}{|S| \cdot \card \calS_{k - 1}} \sum_{S \in \calS_{k - 1}} |E \cap S| \leq \frac{|G_{\calT_{k - 1}} \cap E|}{|S| \cdot \card \calS_{k - 1}}. \end{displaymath}
The denominator $|\calN_{k - 1}| := |S| \cdot \card \calS_{k - 1}$ equals the total measure of the parallelograms in $\calN_{k - 1}$, and this is nearly as large as $|G_{\calT_{k - 1}}|$. More precisely, the difference $|G_{\calT_{k - 1}}| - |\calN_{k - 1}|$ equals the total measure of the parallelograms in $\calE_{k - 1}$, say $|\calE_{k - 1}|$, which is bounded by $C2^{(\lambda-2)k}/k^{2}$. Now, using \eqref{form7} and the trivial estimate $|\calN_{k - 1}| \geq |G_{\calT_{k - 1}}|/2$, we obtain
\begin{displaymath} \frac{|G_{\calT_{k - 1}} \cap E|}{|\calN_{k - 1}|} - \frac{|G_{\calT_{k - 1}} \cap E|}{|G_{\calT_{k - 1}}|} = \frac{|G_{\calT_{k - 1}} \cap E|(|G_{\calT_{k - 1}}| - |\calN_{k - 1}|)}{|G_{\calT_{k - 1}}||\calN_{k - 1}|} < 2r_{k - 1} \frac{|\calE_{k - 1}|}{|G_{\calT_{k - 1}}|} \leq 2Cr_{k - 1}2^{(\lambda-1)k}.  \end{displaymath}
Consequently, and recalling that $r_{k - 1} < \epsilon^{1/3 - \kappa}$, we have
\begin{displaymath} \frac{1}{\card \calS_{k - 1}} \sum_{S \in \calS_{k - 1}} \E_{k}[X_{S}] \leq \frac{|G_{\calT_{k - 1}} \cap E|}{|G_{\calT_{k - 1}}|} + 2Cr_{k - 1}2^{(\lambda-1)k} \leq r_{k - 1} + 2C\epsilon^{1/3 - \kappa}2^{(\lambda-1)k}. \end{displaymath}
Further, recalling that $k \geq k_{\epsilon}$ and $k_{\epsilon}$ grows as $\epsilon$ diminishes, we can assume that $\epsilon$ is so small that $2C\epsilon^{1/3 - \kappa}2^{(\lambda-1)k} < \epsilon^{1/3 - \kappa}/[16(k - k_{\epsilon} + 1)^{2})]$.

After this preparation, we apply \eqref{form11} with $t = \epsilon^{1/3 - \kappa}/[16(k - k_{\epsilon} + 1)^{2}]$:

\begin{align*} \tn_{k} & \left\{\frac{1}{\card \calS_{k - 1}} \sum_{S \in \calS_{k - 1}} \frac{|S \cap G_{\calT_{k}} \cap E|}{|S \cap G_{\calT_{k}}|} \geq r_{k - 1} + \frac{\epsilon^{1/3 - \kappa}}{8(k - k_{\epsilon} + 1)^{2}} \right\}\\
& \leq \tn_{k} \left\{ \frac{1}{\card \calS_{k - 1}} \sum_{S \in \calS_{k - 1}} X_{S} - \frac{1}{\card \calS_{k - 1}} \sum_{S \in \calS_{k - 1}} \E_{k}[X_{S}] \geq \frac{\epsilon^{1/3 - \kappa}}{16(k - k_{\epsilon} + 1)^{2}}  \right\}\\
& \leq \exp\left(-2
\frac{(\card \mathcal{S}_{k-1} )^2}{C2^{\lambda k}\epsilon^{1/3-\kappa}}\cdot
\left[\frac{\epsilon^{1/3 - \kappa}}{16(k - k_{\epsilon} +
1)^{2}}\right]^{2} \right).
\end{align*}

Here
\begin{align*} \card \calS_{k - 1} \sim 2^{\lambda k} = (2^{k_{\epsilon}})^{\lambda}2^{(k - k_{\epsilon})\lambda} = (\epsilon^{-(2 + \kappa)/3})^{\lambda}
 2^{\lambda(k - k_{\epsilon})} &= \epsilon^{-(2/3)\lambda }\epsilon^{- (\kappa/3)\lambda}2^{\lambda(k - k_{\epsilon})}\\
 &\geq \epsilon^{-(2/3)\lambda}2^{(k-k_\epsilon)\lambda}. \end{align*}

Hence, recalling that $\lambda = 1/2$, we find
\begin{align*}
2
\frac{(\card \mathcal{S}_{k-1} )^2}{C2^{\lambda k}\epsilon^{1/3-\kappa}}\cdot
\left[\frac{\epsilon^{1/3 - \kappa}}{16(k - k_{\epsilon} +
1)^{2}}\right]^{2}
&\gtrsim 2^{k/2}\cdot \epsilon^{-1/3}
2^{(k-k_\epsilon)/2}\cdot \frac{1}{2^{k/2} \epsilon^{1/3-\kappa}}
\frac{\epsilon^{2/3-2\kappa}}{16^2 (k-k_\epsilon+1)^4}\\&=
\epsilon^{-
\kappa/2}2^{(k-k_\epsilon)/4} \left[ \epsilon^{-\kappa/2}2^{(k-k_\epsilon)/4}\frac{2}{16^2(k-k_\epsilon+1)^4} \right]. 
\end{align*}
Now, by assuming that $\epsilon > 0$ is small enough (depending on $\kappa$), the bracketed expression on the right hand side above can be made as large as we wish. This gives
\begin{displaymath} \exp\left(-2
\frac{(\card \mathcal{S}_{k-1} )^2}{C2^{\lambda k}\epsilon^{1/3-\kappa}}\cdot
\left[\frac{\epsilon^{1/3 - \kappa}}{16(k - k_{\epsilon} +
1)^{2}}\right]^{2} \right) \leq
\exp\left(-\epsilon^{-\kappa/2} \cdot 2^{(k - k_{\epsilon})/4}
\right). \end{displaymath} 

Combining the efforts so far, we have managed to show that
\begin{equation}\label{form21} \tn_{k} \left\{ D_{\iota_{0}(E)}(G_{k}(\omega)) \geq r_{k - 1} + \frac{\epsilon^{1/3 - \kappa}}{4(k - k_{\epsilon} + 1)^{2}} \right\} \leq \exp\left(-\epsilon^{-\kappa/2} \cdot 2^{(k - k_{\epsilon})/4} \right), \end{equation}
for any fixed isometry $\iota_{0}$. Next, the Continuity lemma \ref{continuityLemma} will be applied in a fairly standard manner to infer the bound

\begin{equation}\label{form15} \tn_{k} \left\{\sup_{\iota} D_{\iota(E)}(G_{k}(\omega)) \geq r_{k} \right\} \leq \epsilon^{3} \cdot 2^{k_{\epsilon} - k}, \end{equation}
valid for $\epsilon > 0$ small enough ({depending only on} $\kappa$).

 The isometries of $\R^{2}$ have the form $\iota(\cdot) = O(\cdot - x)$, where $O$ is an orthogonal transformation (rotation or reflection), and $x \in \R^{2}$. Since $\diam E \leq 2$, and all the sets $G_{k}(\omega)$ lie inside $[0,1]^{2}$, the vectors $x$ such that
\begin{equation}\label{form19} G_{k}(\omega) \cap \iota(E) = G_{k}(\omega) \cap O(E - x) \neq \emptyset \end{equation}
for some orthogonal transformation $O$ must lie inside some fixed ball $B(z,10)$, $z \in E$. For a suitable parameter $\delta > 0$, we choose a $\delta$-net $\{x_{1},\ldots,x_{n}\}$ inside $B(z,10)$; then $n \sim \delta^{-2}$. We also choose a $\delta$-net of orthogonal transformations $\{O_{1},\ldots,O_{m}\}$: this simply means that $\min\{ \|O - O_{j}\| : 1 \leq j \leq n\} \leq \delta$ for any orthogonal transformation $O$. Such a $\delta$-net can be found with cardinality $m \lesssim \delta^{-1}$. Consequently, the family of isometries $\iota_{ij}(\cdot) := O_{i}(\cdot - x_{j})$ has cardinality $mn \lesssim \delta^{-3}$. Moreover, if $\iota(\cdot) = O(\cdot - x)$ is any isometry with the property \eqref{form19}, then $|x - x_{j}| \leq \delta$ and $\|O - O_{i}\| \leq \delta$ for some $i,j$. It follows that
\begin{equation}\label{form20} \min\{\|\iota - \iota_{ij}\| : 1 \leq i \leq m \text{ and } 1 \leq j \leq n\} \lesssim \delta, \quad \text{whenever } \iota \text{ satisfies \eqref{form19}.} \quad \end{equation}

For any $\omega \in \Omega$, the set $G_{k}(\omega)$ is a union of $2^{k}$ parallelograms in a certain family $\calT_{k}$. Hence, by the Continuity lemma \ref{continuityLemma}, for two isometries $\iota_{1},\iota_{2}$, we have
\begin{align*} D_{\iota_{1}(E)}(G_{k}(\omega)) -D_{\iota_{2}}(G_{k}(\omega)) & = \frac{1}{|G_{k}(\omega)|} \sum_{T \in \calT_{k}} (|T \cap \iota_{1}(E)| - |T \cap \iota_{2}(E)|)\\
& \lesssim n_{k} \sum_{T \in \calT_{k}} 2^{-k}\|\iota_{1} - \iota_{2}\| = n_{k}\|\iota_{1} - \iota_{2}\|. \end{align*}
It follows from this and \eqref{form20} that if
\begin{equation}\label{form12} \max_{i,j} D_{\iota_{ij}(E)}(G_{k}(\omega)) < r_{k - 1} + \frac{\epsilon^{1/3 - \kappa}}{4(k - k_{\epsilon} + 1)^{2}}, \end{equation}
and $n_{k}\delta < \epsilon^{1/3 - \kappa}/[C(k - k_{\epsilon} +
1)^{2}]$ for some suitable constant $C \geq 1$, then in fact
\begin{displaymath} \sup_{\iota} D_{\iota(E)}(G_{k}(\omega))  < r_{k - 1} + \frac{\epsilon^{1/3 - \kappa}}{2(k - k_{\epsilon} + 1)^{2}} = r_{k}. \end{displaymath}
We now let $\delta = \epsilon^{1/3 - \kappa}/[2Cn_{k}(k -
k_{\epsilon} + 1)^{2}] \gtrsim \epsilon^{1/3}/[k^{4} \cdot 2^{k}]$
be small enough for this purpose. Then, recalling that
$2^{k_{\epsilon}} = \epsilon^{-(2 + \kappa)/3} \geq
\epsilon^{-2/3}$, the family of isometries $\iota_{ij}$ has
cardinality at most
\begin{displaymath} mn \lesssim \delta^{-3} \lesssim k^{12} \cdot 2^{3k} \cdot \epsilon^{-1} \lesssim 2^{6k}, \qquad k > k_{\epsilon}, \end{displaymath}
and the probability that \eqref{form12} should fail for even one
of these $\iota_{ij}$ is hence bounded by $\sim 2^{6k}$ times the
bound from \eqref{form21}:
\begin{displaymath} \tn_{k} \left\{ \max_{i,j} D_{\iota_{ij}(E)}(G_{k}(\omega)) \geq r_{k - 1} + \frac{\epsilon^{1/3 - \kappa}}{4(k - k_{\epsilon} + 1)^{2}} \right\} \lesssim 2^{6k} \exp\left(-\epsilon^{-\kappa/2} \cdot 2^{(k - k_{\epsilon})/4} \right). \end{displaymath}
By the discussion around \eqref{form12}, this implies that
\begin{align*} \tn_{k} \left\{ \sup_{\iota} D_{\iota(E)}(G_{k}(\omega)) \geq r_{k} \right\} & \lesssim 2^{6k} \exp\left(-\epsilon^{-\kappa/2} \cdot 2^{(k - k_{\epsilon})/4} \right)\\
& = 2^{6k_{\epsilon}}2^{6(k -
k_{\epsilon})}\exp\left(-\epsilon^{-\kappa/2} \cdot 2^{(k -
k_{\epsilon})/4} \right). \end{align*} Finally, recall once more
that $\epsilon^{-\kappa} = 2^{k_{\epsilon}(3\kappa/(2 +
\kappa))}$, and observe that
\begin{displaymath} \sup_{x \geq 1} C^{R_{1}}x^{R_{2}}\exp(-C^{r_{1}}x^{r_{2}}) \to 0, \quad \text{as } C \to \infty \end{displaymath}
for any (fixed) choices of $R_{1},R_{2} \geq 1$ and $r_{1},r_{2} > 0$. This implies 
\begin{displaymath}
\epsilon^{-3} 2^{k-k_\epsilon} \tn_{k} \left\{\sup_{\iota} D_{\iota(E)}(G_{k}(\omega)) \geq r_{k} \right\}\leq 1,
\end{displaymath}
and thus \eqref{form15}, for small enough $\epsilon > 0$ (depending only on $\kappa$).

Recalling the discussion leading to \eqref{form13}, we have managed to prove that
\begin{displaymath} \tn\{g_{k - 1}(\omega) < r_{k - 1} \text{ and } g_{k}(\omega) \geq r_{k} \} \leq \epsilon^{3} \cdot 2^{k_{\epsilon} - k} \end{displaymath}
for $\epsilon > 0$ small enough, and it then follows from \eqref{form14} that
\begin{displaymath} \tn\{\sup_{k} g_{k}(\omega) \geq \epsilon^{1/3-\kappa}\} \leq \sum_{k = k_{\epsilon} + 1}^{\infty} \epsilon^{3} \cdot 2^{k_{\epsilon} - k} = \epsilon^{3}. \end{displaymath}
This completes the proof of the lemma. \end{proof}

The statement of the next lemma is very similar to the previous one, except that $D_{\iota(E)}(G_{k}(\omega))$ has been replaced by $\calH^{1}(G(\omega) \cap \iota(E))$:
\begin{lemma}[First intersection lemma]\label{densityLemmaTwo} Fix $\epsilon, \kappa > 0$, and let $E \subset \R^{2}$ be a Borel set with $|E| < \epsilon$ and $\diam E \leq 2$. Then,
\begin{displaymath} \tn\{\sup_{\iota} \calH^{1}(G(\omega) \cap \iota(E)) > \epsilon^{1/3 - \kappa} \} \leq \epsilon^{2}, \end{displaymath}
if $\epsilon > 0$ is small enough, depending only on $\kappa$.
\end{lemma}

\begin{proof}

Cover $E$ with a sequence of dyadic squares $Q_{1},Q_{2},\ldots$ such that $\sum |Q_{j}| < \epsilon$. Then, let
\begin{displaymath} E_{n} := \bigcup_{j = 1}^{n} Q_{j}. \end{displaymath}
For a fixed $n \in \N$, the set $E_{n}$ is a finite union of dyadic squares, so it can be covered by a finite union dyadic squares of a fixed side-length $\delta_{n}$ without altering the total measure $|E_{n}| < \epsilon$. We do this, but for convenience we continue to denote the constituent squares of $E_{n}$ by $Q_{1},\ldots,Q_{n}$.

For each $\omega \in \Omega$ such that $\sup_{\iota}
\calH^{1}(G(\omega) \cap \iota(E)) > \epsilon^{1/3 - \kappa}$,
there exists $n(\omega) \in \N$ such that
\begin{displaymath} \sup_{\iota} \calH^{1}(G(\omega) \cap \iota(E_{n})) > \epsilon^{1/3 - \kappa}, \quad n \geq n(\omega). \end{displaymath}
In particular, if the claim of the lemma fails for some small $\epsilon > 0$, there exists an integer $n = n(\epsilon) \in \N$ such that
\begin{equation}\label{form16} \tn\left\{\sup_{\iota} \calH^{1}(G(\omega) \cap \iota(E_{n})) > \epsilon^{1/3 - \kappa}\right\} > \epsilon^{2}. \end{equation}

Recall that $E_{n}$ consists of squares of side-length $\delta_{n}$, and pick $k \in \N$ so large that $2^{-k} < \delta_{n}$. Then, for any $\omega \in \Omega$ in the event displayed in \eqref{form16}, we claim that
\begin{equation}\label{form17} \sup_{\iota} D_{\iota(E_{n})}(G_{k}(\omega)) \geq c\epsilon^{1/3 - \kappa}. \end{equation}
where $c> 0$ is an absolute constant. To see this, first pick an
isometry $\iota$ such that $\calH^{1}(G(\omega) \cap \iota(E_{n}))
> \epsilon^{1/3 - \kappa}$, and assume for convenience that $\iota
= \operatorname{Id}$. Let $\calT_{k}$ be the the collection of
parallelograms constituting $G_{k}(\omega)$, and observe that
\begin{displaymath} \epsilon^{1/3 - \kappa} < \calH^{1}(G(\omega) \cap E_{n}) \lesssim 2^{-k} \card\{T \in \calT_{k} : T \cap E_{n} \neq \emptyset\}, \end{displaymath}
because the base width of each $T$ is $2^{-k}$, and $G(\omega)$ is
a $1$-Lipschitz graph. In other words, at least $\gtrsim
2^{k}\epsilon^{1/3 - \kappa}$ rectangles $T \in \calT_{k}$ meet at
least one square $Q_{j(T)}$, $1 \leq j(T) \leq n$, see Figure
\ref{fig4}.
\begin{figure}[h!]
\begin{center}
\includegraphics[scale = 0.4]{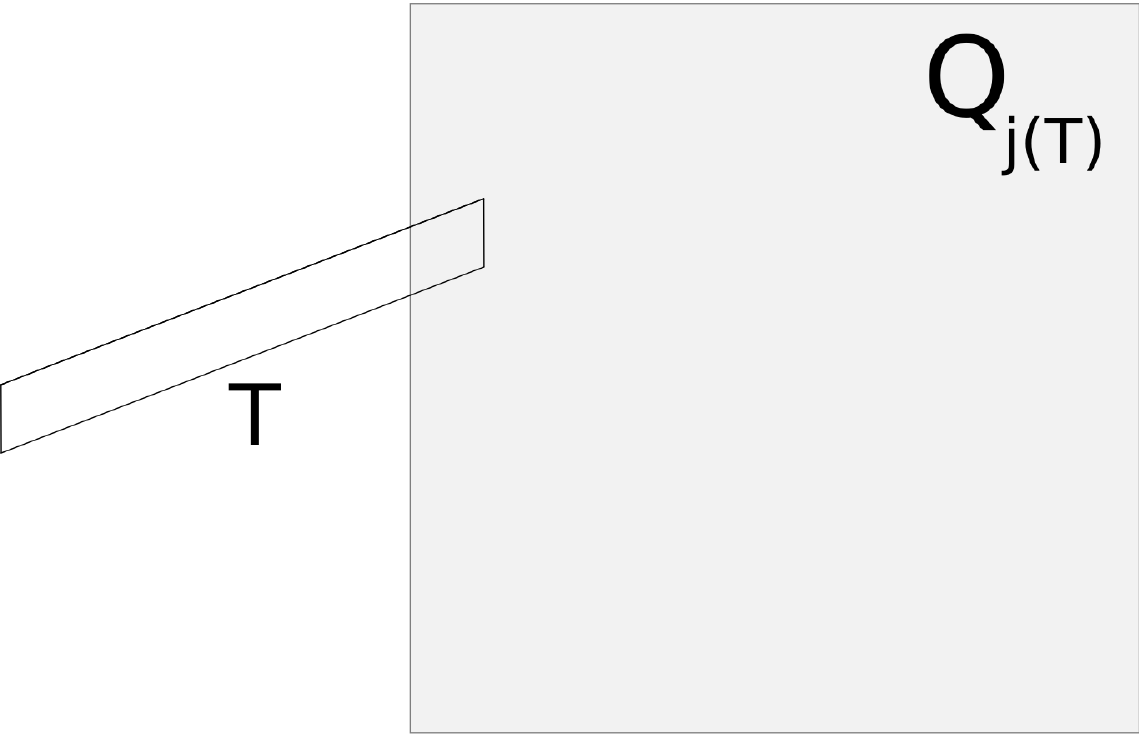}
\caption{A parallelogram $T \in \calT_{k}$ intersecting the square
$Q_{j(T)} \subset E_{n}$.}\label{fig4}
\end{center}
\end{figure}
Now, since $2^{-k} < \delta_{n}$, the parallelogram $T$ would fit entirely inside $Q_{j(T)}$, and it is easy to see the following: there is a fixed family of unit vectors $\{e_{1},\ldots,e_{100}\} \subset S^{1}$ such that
\begin{displaymath} \max_{1 \leq j \leq 100} \frac{|T \cap [Q_{j(T)} + 2^{-k}e_{j}]|}{|T|} \geq c, \end{displaymath}
where $c > 0$ is an absolute constant. It follows that there is a
fixed vector $x_{0} := 2^{-k}e_{j_{0}}$, and $\gtrsim
2^{k}\epsilon^{1/3 - \kappa}$ parallelograms $T \in \calT_{k}$
such that for all these $T$,
\begin{displaymath} \frac{|T \cap [Q_{j(T)} + x_{0}]|}{|T|} \geq c. \end{displaymath}
Consequently, denoting the common size of the rectangles $T \in \calT_{k}$ by $|T|$,
\begin{displaymath} |G_{k}(\omega) \cap (E_{n} + x_{0})| \gtrsim {2^{k}\epsilon^{1/3 - \kappa}} \cdot c|T| = {c\epsilon^{1/3 - \kappa}}|G_{k}(\omega)|, \end{displaymath}
or $D_{E_{n} + x_{0}}(G_{k}(\omega)) \gtrsim c\epsilon^{1/3 -
\kappa}$. This proves \eqref{form17} and shows that
\begin{displaymath} \tn \left\{ \sup_{\iota} \calH^{1}(G(\omega) \cap \iota(E_{n})) > \epsilon^{1/3 - \kappa}\right\} \leq \tn \left\{\sup_{k \in \N} \sup_{\iota} D_{\iota(E_{n})}(G_{k}(\omega)) \geq c\epsilon^{1/3 - \kappa}\right\}. \end{displaymath}
But, writing $c\epsilon^{1/3 - \kappa} = [c^{3/(1 -
3\kappa)}\epsilon]^{1/3 - \kappa}$, Lemma \ref{densityLemmaOne}
says that the latter probability is bounded by $(c^{3/(1 -
3\kappa)}\epsilon)^{3}$ as soon as $\epsilon > 0$ is small enough,
depending only on $\kappa$. For $\epsilon > 0$ small enough, this
contradicts \eqref{form16} and completes the proof. \end{proof}

The next lemma is the same as the previous one without the assumption $\diam E \leq 2$:
\begin{lemma}[Second intersection lemma]\label{densityLemmaThree} Fix $\epsilon,\kappa > 0$, and let $E \subset \R^{2}$ be a Borel set with $|E| < \epsilon$. Then
\begin{displaymath} \tn\left\{\sup_{\iota} \calH^{1}(G(\omega) \cap \iota(E)) > \epsilon^{1/3 - \kappa} \right\} \leq \epsilon, \end{displaymath}
if $\epsilon > 0$ is small enough, depending only on $\kappa$.
\end{lemma}

\begin{proof} Let $\calD$ be the collection of dyadic squares of side-length $1$, and write
\begin{displaymath} \calD_{j} := \{Q \in \calD : \epsilon 2^{-j - 1} < |E \cap Q| \leq \epsilon 2^{-j}\}, \qquad j \geq 0. \end{displaymath}
Then
\begin{equation}\label{form18} \card \calD_{j} < 2^{j + 1}, \end{equation}
and $\diam (E \cap Q) \leq 2$ for any $Q \in \calD$. Write
$\calD_{\infty}$ for those $Q \in \calD$ with $|Q \cap E| = 0$.
Recalling that $G(\omega) \subset [0,1]^{2}$ for all $\omega \in
\Omega$ gives the following observation: if $\calH^{1}(G(\omega)
\cap \iota(E)) > \epsilon^{1/3 - \kappa}$ for some $\omega \in
\Omega$ and isometry $\iota$, then $\calH^{1}(G(\omega) \cap
\iota(E \cap Q)) > \epsilon^{1/3 - \kappa}/10$ for some $Q \in
\calD$. Fixing $j \geq 0$, Lemma \ref{densityLemmaTwo}  and
\eqref{form18} imply that
\begin{displaymath} \tn\left\{ \sup_{\iota} \calH^{1}(G(\omega) \cap ([E \cap Q] + x)) > \epsilon^{1/3 - \kappa}/10 \text{ for some } Q \in \calD_{j} \right\} \lesssim 2^{j}(\epsilon 2^{-j})^{2} = \frac{\epsilon^{2}}{2^{j}}\end{displaymath}
for all $\epsilon > 0$ small enough (depending only on $\kappa >
0$); in fact we could even replace $\epsilon^{1/3 - \kappa}$ by
$(\epsilon 2^{-j})^{1/3 - \kappa}$ and still have the same bound.
For $j = \infty$, the probability above is just zero, as it is
clearly zero for each individual $Q \in \calD_{\infty}$.

Combining everything,
\begin{displaymath} \tn\left\{ \sup_{\iota} \calH^{1}(G(\omega) \cap \iota(E \cap Q)) > \epsilon^{1/3 - \kappa}/10 \text{ for some } Q \in \calD \right\} \lesssim \sum_{j} \frac{\epsilon^{2}}{2^{j}} \sim \epsilon^{2}, \end{displaymath}
and so the claim of the lemma certainly holds for $\epsilon > 0$ small enough.
\end{proof}

\section{Proof of the main theorem}\label{mainProof}

We quickly recall the main result, and then prove it.
\begin{thm} Let $\Gamma$ be a family of curves, which contains some isometric copy of every set of the form
\begin{displaymath} G_{f} := \{(x,f(x)) : x \in [0,1]\}, \end{displaymath}
where $f \colon [0,1] \to [0,1]$ is $1$-Lipschitz. Then
$\mathrm{mod}_{p}(\Gamma) \geq c_{p} > 0$ for every $p > 3$.
\end{thm}
\begin{proof}
We assume that $\Gamma$ contains an isometric copy of every graph
$G(\omega)$, $\omega \in \Omega$. We then make a counter
assumption: fixing $p > 3$, assume $\mathrm{mod}_p(\Gamma) < c$
for some small constant $c > 0$. Then, there exists a
$\Gamma$-admissible Borel function $\rho \colon \R^{2} \to
[0,\infty]$ such that
\begin{displaymath} \int \rho^{p}(x) \, dx < c. \end{displaymath}
Write
\begin{displaymath} \rho := \sum_{j = 0}^{\infty} \rho \chi_{E_{j}}, \end{displaymath}
where $E_{0} := \{x : 0 \leq \rho(x) < 1/2\}$, and $E_{j} := \{x : 2^{j - 2} \leq \rho(x) < 2^{j - 1}\}$ for $j \geq 1$. For every graph $G(\omega)$, $\omega \in \Omega$, the family $\Gamma$ contains some isometric copy $\gamma_{\omega} := \iota_{\omega}(G(\omega))$. Since $\rho$ is admissible, we have
\begin{displaymath}1 \leq \int_{\gamma_{\omega}} \rho \, d\calH^{1} \leq \calH^{1}(\gamma_{\omega})/2 + \frac{1}{2} \sum_{j = 1}^{\infty} 2^{j} \calH^{1}(\gamma_{\omega} \cap E_{j}) \leq \frac{1}{2} + \frac{1}{2} \sum_{j = 1}^{\infty} 2^{j}\calH^{1}(\gamma_{\omega} \cap E_{j}). \end{displaymath}
In particular, for every $\omega \in \Omega$, we ought to have $\calH^{1}(\gamma_{\omega} \cap E_{j}) \geq 2^{-j}$ for some $j$. So, to produce a contradiction, it suffices to find a graph $G(\omega)$, $\omega \in \Omega$, such that
\begin{equation}\label{form22} \sup_{j \geq 1} \sup_{\iota} 2^{j} \cdot \calH^{1}(E_{j} \cap \iota(G(\omega)) < 1. \end{equation}

To do this, first observe that
\begin{displaymath} |E_{j}| \leq (2^{j - 2})^{-p} \int_{E_{j}} \rho^{p}(x) \, dx \leq c \cdot (2^{j - 2})^{-p}, \quad j \geq 1. \end{displaymath}
For $c$ small enough (depending only on $p$), Lemma \ref{densityLemmaThree} implies that
\begin{displaymath} \tn\left\{ \sup_{\iota} \calH^{1}(E_{j} \cap \iota(G(\omega))) \geq c^{1/p}2^{2 - j} \right\} \leq c \cdot (2^{j - 2})^{-p}. \end{displaymath}
In particular, if $c$ is, in addition, so small that $c^{1/p}2^{2 - j} \leq 2^{-j}$, we have
\begin{displaymath} \tn\left\{\sup_{\iota} \calH^{1}(E_{j} \cap \iota(G(\omega))) \geq 2^{-j} \right\} \leq c \cdot (2^{j - 2})^{-p}. \end{displaymath}
Finally, we choose $c$ so small the upper bounds $c \cdot (2^{j - 2})^{-p}$ sum up to something strictly less than one. This guarantees the existence of a graph $G(\omega)$ such that \eqref{form22} holds, and the ensuing contradiction gives a lower bound for $c$, which only depends on $p$. \end{proof}

\section{Proof of Corollary \ref{mainCor} and Further remarks} We recall the statement of Corollary \ref{mainCor}:
\begin{cor} Let $\delta \in (0,1]$, and associate to every length-$1$ rectifiable curve $\gamma$ in $\R^{2}$ an $\calH^{1}$-measurable
subset $E_{\gamma}$ of length at least $\delta$, and an isometry $\iota_{\gamma}$.
 Then, the union of the sets $\iota_{\gamma}(E_{\gamma})$ has Lebesgue outer measure at least $\gtrsim c_{p}\delta^{p}$ for any $p > 3$. \end{cor}

\begin{proof} Denote the said union by $K$, and assume that $K$ has finite Lebesgue outer measure. Cover $K$ by squares with total area at most $2|K|$, and write $\tilde{K}$ for the union of these squares. Then $\rho := \delta^{-1}\chi_{\tilde{K}} \in \mathrm{adm}(\Gamma)$, where $\Gamma$ is the family of curves $\Gamma := \{\iota_{\gamma}(\gamma) : \calH^{1}(\gamma) = 1\}$. Since the Lipschitz graphs $G_{f}$ in Theorem \ref{main} have length at least one, every one of them contains some sub-curve $\gamma_{f}$ of length exactly one, and we can consider the associated isometries $\iota_{f} := \iota_{\gamma_{f}}$; then $\rho$ is admissible for the family $\{\iota_{f}(G_{f}) : f \colon [0,1] \to [0,1] \text{ is } 1\text{-Lipschitz}\}$, and Theorem \ref{main} with $p > 3$ implies that
\begin{displaymath} |K| \geq \frac{|\tilde{K}|}{2} = \frac{\delta^{p}}{2} \int \rho^{p} \, dx \geq \frac{c_{p}\delta^{p}}{2}. \end{displaymath}
This completes the proof. \end{proof}

\begin{remark} It would be interesting to know $\inf \{p \geq 2 : \mathrm{mod}_{p}(\Gamma) > 0\}$
for every Moser family $\Gamma$, and we strongly suspect that $3$
is not the answer. For instance, using the technique of the paper,
it is not hard to show the following: if $\Gamma$ is a family of
sets containing an isometric copy of every $1$-Ahlfors-David
regular set in $[0,1]^{2}$ (with regularity constants bounded by
$10$, say), then $\Gamma$ has positive $p$-modulus for every $p >
2$. Here modulus is defined in the obvious way, with $\rho \in
\textrm{adm}(\Gamma)$, if $\int_{K} \rho \, d\calH^{1} \geq 1$ for
every set $K \in \Gamma$. The proof is simpler than that of
Theorem \ref{main}, mainly because it is easier to construct
random $1$-Ahlfors-David regular sets than random graphs
(subdivide $[0,1]^{2}$ into four sub-squares and select two at
random; then subdivide the remaining squares into four pieces,
select two at random inside each, and continue \emph{ad
infinitum}). From a technical point of view, the improvement from
$3$ to $2$ is caused by the fact that there is no longer a need
for "exceptional sets", and one can prove an analogue of the
Second intersection lemma with $\epsilon^{1/3 - \kappa}$ replaced
by $\epsilon^{1/2 - \kappa}$.

We conjecture that $\inf \{p \geq 2 : \mathrm{mod}_{p}(\Gamma) > 0\} = 2$ for every Moser family $\Gamma$. \end{remark}

\bibliographystyle{amsplain}
\bibliography{referencesMod}

\def\cprime{$'$}
\providecommand{\bysame}{\leavevmode\hbox to3em{\hrulefill}\thinspace}
\providecommand{\MR}{\relax\ifhmode\unskip\space\fi MR }
\providecommand{\MRhref}[2]{%
  \href{http://www.ams.org/mathscinet-getitem?mr=#1}{#2}
}
\providecommand{\href}[2]{#2}
\begin{thebibliography}{10}

\bibitem{MR1544912}
A.~S. Besicovitch, \emph{On {K}akeya's problem and a similar one}, Math. Z.
  \textbf{27} (1928), no.~1, 312--320. \MR{1544912}

\bibitem{MR0229779}
A.~S. Besicovitch and R.~Rado, \emph{A plane set of measure zero containing
  circumferences of every radius}, J. London Math. Soc. \textbf{43} (1968),
  717--719. \MR{0229779 (37 \#5345)}

\bibitem{MR2163782}
Peter Brass, William Moser, and J{\'a}nos Pach, \emph{Research problems in
  discrete geometry}, Springer, New York, 2005. \MR{2163782 (2006i:52001)}

\bibitem{MR0297953}
Roy~O. Davies, \emph{Another thin set of circles}, J. London Math. Soc. (2)
  \textbf{5} (1972), 191--192. \MR{0297953 (45 \#7005)}

\bibitem{MR1800917}
Juha Heinonen, \emph{Lectures on analysis on metric spaces}, Universitext,
  Springer-Verlag, New York, 2001. \MR{1800917 (2002c:30028)}

\bibitem{MR0144363}
Wassily Hoeffding, \emph{Probability inequalities for sums of bounded random
  variables}, J. Amer. Statist. Assoc. \textbf{58} (1963), 13--30. \MR{0144363}

\bibitem{MR1535168}
J.~R. Kinney, \emph{A {T}hin {S}et of {C}ircles}, Amer. Math. Monthly
  \textbf{75} (1968), no.~10, 1077--1081. \MR{1535168}

\bibitem{MR557120}
J.~M. Marstrand, \emph{Packing smooth curves in {${\bf R}^{q}$}}, Mathematika
  \textbf{26} (1979), no.~1, 1--12. \MR{557120 (81d:52009)}

\bibitem{MR2466579}
Olli Martio, Vladimir Ryazanov, Uri Srebro, and Eduard Yakubov, \emph{Moduli in
  modern mapping theory}, Springer Monographs in Mathematics, Springer, New
  York, 2009. \MR{2466579 (2012g:30004)}

\bibitem{MR1333890}
Pertti Mattila, \emph{Geometry of sets and measures in {E}uclidean spaces},
  Cambridge Studies in Advanced Mathematics, vol.~44, Cambridge University
  Press, Cambridge, 1995, Fractals and rectifiability. \MR{1333890 (96h:28006)}

\bibitem{MR664434}
W.~O.~J. Moser, G.~Blind, V.~Klee, C.~Rousseau, J.~Goodman, B.~Monson,
  J.~Wetzel, L.~M. Kelly, G.~Purdy, and J.~Wilker (eds.), \emph{Problems in
  discrete geometry, 1980}, fifth ed., William Moser, Department of
  Mathematics, McGill University, Montreal, Que., 1980. \MR{664434 (84c:51003)}

\bibitem{MR1961007}
Rick Norwood and George Poole, \emph{An improved upper bound for {L}eo
  {M}oser's worm problem}, Discrete Comput. Geom. \textbf{29} (2003), no.~3,
  409--417. \MR{1961007 (2004a:52033)}

\bibitem{MR0264035}
D.~J. Ward, \emph{A set of plane measure zero containing all finite polygonal
  arcs}, Canad. J. Math. \textbf{22} (1970), 815--821. \MR{0264035 (41 \#8634)}

\end{thebibliography}

\end{document}